\documentclass[J,11pt]{cmsart}
\Volume{3}
\Firstpage{1}         
\ReceivedRevised{May 17, 2022}{}

\usepackage{amssymb,amstext,amsmath,amscd,amsthm,enumerate,latexsym}
\usepackage[bbgreekl]{mathbbol}
\Firstpage{1}                       
\usepackage{lineno,hyperref}
\modulolinenumbers[1]
\theoremstyle{plain}

\usepackage{color}
\usepackage[all]{xy}
\usepackage{imakeidx}
\usepackage{cases}
\usepackage{graphicx}
\usepackage{enumitem}
\usepackage{enumitem}
\usepackage{indentfirst}
\usepackage{hyperref}
\usepackage{graphicx,color}
\numberwithin{equation}{section}

\DeclareSymbolFontAlphabet{\mathbb}{AMSb}
\DeclareSymbolFontAlphabet{\mathbbl}{bbold}

\newtheorem{thm}{Theorem}[section]

\newtheorem{lem}[thm]{Lemma}

\newtheorem{prop}[thm]{Proposition}
\theoremstyle{definition}
\newtheorem{defn}[thm]{Definition}
\theoremstyle{remark}
\newtheorem{remark}[thm]{Remark}

\theoremstyle{plain}
\numberwithin{equation}{section}      

\let\lvert=|\let\rvert=|

\begin{document}

\title{Dynamical Systems and $\lambda$-Aluthge transforms}

\author[Linh Tran]{Linh Tran}
\address{*\\
         Department of Mathematics\\
         Chungnam  National University\\
         Daejon 305-764, Republic of Korea}
\email{tolinh.tran91@gmail.com}             



\keywords{Aluthge transform, quasi-hyperbolic, $\lambda$-Aluthge transform, bounded shadowing property}

\subjclass{Primary: 47A15; Secondary: 47B49, 47B37, 37C20}

\begin{abstract}
In this note, we verify that the bounded shadowing property and quasi-hyperbolicity of bounded linear operators on Hilbert spaces are preserved under Aluthge transforms.
\end{abstract}
\maketitle

\section{Introduction}
Let $\left( \mathbb{H}, \left< . \right> \right)$ be a Hilbert space and $B \left( \mathbb{H} \right)$ be the algebra of bounded linear operators on $\mathbb{H}$. For each $T \in B \left( \mathbb{H} \right)$, the polar decomposition of $T$ is uniquely defined by
\begin{linenomath}
\begin{align*}
    T = U \left| T \right|
\end{align*}
\end{linenomath}
in which $\left| T \right| = \left( T^*T \right) ^{\frac{1}{2}}$ and $\ker \left( U \right) = \ker \left( \left| T \right| \right) = \ker \left( T \right)$. Without further mention, the spectrum of $T$ is
\begin{linenomath}
\begin{align*}
    \sigma \left( T \right) = \left \{ \lambda \in \mathbb{C} \mid \nexists \left( T -\lambda I \right) ^{-1} \in B \left( \mathbb{H} \right) \right \}.
\end{align*} 
\end{linenomath}
 During the last 30 years, Operator Theorists have paid their attention to Aluthge transforms because they convert the initial operators into one which are closer to normal operators. I. B. Jung \cite[Theorem 2.1]{IBJ2} showed that
\begin{linenomath}
\begin{align*}
    \sigma \left( T \right) = \sigma \left( \Delta \left( T \right) \right)
\end{align*}
\end{linenomath}
where $\sigma \left( . \right)$ are the corresponding spectra. Moreover, Jung-Ko-Pearcy \cite[Corollary 1.16]{IBJ1} proved that $T$ has nontrivial invariant subspaces if and only if so does $\Delta \left( T \right)$.\\
On the other sides, another reason relates with the Aluthge iterates defined by
\begin{linenomath}
\begin{align*}
    \begin{cases}
\Delta ^{(0)} \left( T \right) = T \\ 
\Delta ^{(n)} \left( T \right) = \Delta \left( \Delta ^{(n-1)} \left( T \right) \right), \quad\quad\forall n \ge 1,  n \in \mathbb{Z}^{+}.
\end{cases} 
\end{align*}
\end{linenomath}
The Aluthge iterates of bounded linear operators are convergent to normal operators for $\dim \left( \mathbb{H} \right) < \infty$ that has been verified in \cite[Theorem 4.2]{Antezana1} but not in general as shown in \cite[Corollary 3.2]{IBJ2}.\\We have the following remark.
\begin{remark}
\label{remark homogenous of Aluthge transform}
Let $T = U \left| T \right| \in B \left( \mathbb{H} \right)$ be the polar decomposition of $T$.
\begin{enumerate}[label=(\roman*)]
    \item For every $\alpha \in \mathbb{C}$, $\Delta \left( \alpha T \right) = \left| \alpha \right| \Delta \left( T \right)$;
    \item \cite{IBJ1} $T$ is quasinormal if and only if $U \left| T \right| = \left| T \right| U$. \\
    Hence, $\Delta \left( T \right) = T$;
    \item \cite{Yamazaki1} $\lim_{n \rightarrow \infty} \left \| \Delta ^{\left( n \right)} \left( T \right) \right \| = r \left( T \right)$.
    \end{enumerate}
\end{remark}
Regarding to the spectral properties, we have the following remark.
\begin{remark} \cite[Theorem 1.3]{IBJ1}
\label{remark spectral property}
Let $T \in B \left( \mathbb{H} \right)$ be a bounded linear operator with the polar decomposition $T = U \left| T\right|$ and $\Delta \left( T \right) = \left| T \right|^{\frac{1}{2}}U \left| T \right|^{\frac{1}{2}}$ be the Aluthge transform of $T$. Then the following properties hold.
\begin{enumerate}[label=(\roman*)]
    \item $\sigma \left( T \right) = \sigma \left( \Delta \left( T \right) \right)$;
    \item $\sigma_{ap} \left( T \right) = \sigma_{ap} \left( \Delta \left( T \right) \right)$;
    \item $\sigma _{p} \left( T \right) = \sigma _{p} \left( \Delta \left( T \right) \right)$.
\end{enumerate} 
where $\sigma_{ap} \left( . \right)$ (resp. $\sigma_{p} \left( . \right)$) are the corresponding approximate point (resp. point) spectra.
\end{remark}
Aluthge transforms and Aluthge iterates of bounded linear operators $T$ on Hilbert space $\mathbb{H}$ were generalized to $\lambda$-Aluthge transforms and $\lambda$-Aluthge iterates.
\begin{defn}
\label{def lambda Aluthge transform and lambda aluthge iterates}
Let $T = U \left| T \right| \in B \left( \mathbb{H} \right)$ be the polar decomposition of $T$. For every $\lambda \in \left( 0,1 \right)$, the $\lambda$-Aluthge transform of $T$ is defined by
\begin{linenomath}
\begin{align*}
    \Delta _{\lambda} \left( T \right) = \left| T \right| ^{\lambda} U \left| T \right| ^{1-\lambda}.
\end{align*}
\end{linenomath}
For each $\lambda \in \left( 0,1 \right)$, the $\lambda$-Aluthge iterates of $T$ is determined by
\begin{linenomath}
\begin{align*}
    \begin{cases}
\Delta_{\lambda}^{(0)} \left( T \right) = T \\ 
\Delta_{\lambda} ^{(n)} \left( T \right) = \Delta_{\lambda} \left( \Delta_{\lambda} ^{(n-1)} \left( T \right) \right), \quad\quad\forall n \ge 1,  n \in \mathbb{Z}^{+}.
\end{cases} 
\end{align*}
\end{linenomath}
\end{defn}

On the other hand, dynamicists recently published various results related to the behaviors of Dynamical Systems under Aluthge transforms such as in \cite{Bernades1, Bernades2, Cirilo1}. In fact, I. B. Jung et al. verified that $T$ and its Aluthge transform $\Delta \left( T \right)$ are similar if $T$ is invertible in \cite[Lemma 1.1]{IBJ1}. Therefore, they share shadowing property, topological transitivity and chaotic behaviors. In addition, K. Lee and C. A. Morales introduced bounded shadowing property in \cite{lm}.
\begin{defn}
\label{defn bounded shadowing property}
A homeomorphism $g: Y \longrightarrow Y$ has the bounded shadowing property if for every $\epsilon >0$, there is $\delta >0$ such that every bounded $\delta$-pseudo orbit can be $\epsilon$-shadowed.
\end{defn}
Actually, every shadowing operators has bounded shadowing property but the converse does not held. In the light of these results, our first result verifies that bounded shadowing property is preserved under $\lambda$-Aluthge transform in Theorem \ref{theorem bounded shadowing is preserved}.\\
In addition, using the spectral properties of $\lambda$-Aluthge transform (for every $\lambda \in \left( 0,1 \right)$), we show that quasi-hyperbolicity is invariant under the action of $\lambda$-Aluthge transforms in Proposition \ref{prop quasihyperbolic is preserved}.
\section{Main results}
Let us begin this section by recalling that on metric spaces $\left( X, d_1 \right)$ and $\left( Y, d_2 \right)$, a homeomorphism $L: \left( X, d_1 \right) \longrightarrow \left( Y, d_2 \right)$ is called Lipeomorphism if $L$ satisfies Lipschitz condition, i.e for every $x,y \in X$, there are constants $A,B>0$ such that
\begin{align*}
    Ad_1 \left( x,y \right) \le d_2 \left( L \left( x \right), L \left( y \right) \right) \le B d_1 \left( x,y \right).
\end{align*}
More precisely, $L$ is Lipeomorphism if and and only there is a constant $L>0$ so that
\begin{align*}
    d_2 \left( L \left( x \right) , L \left( y \right) \right) \le L d_1 \left( x,y \right).
\end{align*}
Lipeomorphisms draw the attention of mathematicians because they preserve boundedness and completeness of the domain. Furthermore, Lee-Morales verify that bounded shadowing property is invariant under certain topological conjugacies as shown in the following lemma.
\begin{lem} \cite[Lemma 6]{lm}
\label{lem lipeomorphism}
Let $Y$, $Z$ be metric spaces and $H: Z \longrightarrow Y$ be a Lipeomorphism (i.e Lipschitz homeomorphism with Lipschitz inverse). If $P: Z \longrightarrow Z$ is a homeomorphism with the bounded shadowing property, then so does $H \circ P \circ H^{-1}: Y \longrightarrow Y$.
\end{lem}
As a consequence, we have our first result.
\begin{thm}
\label{theorem bounded shadowing is preserved}
Let $T = U \left| T \right| \in B \left( \mathbb{H} \right)$ be the polar decomposition of invertible bounded linear operator $T$. Then $T$ has bounded shadowing property if and only if $\Delta \left( T \right)$ does.
\end{thm}
\begin{proof}
Let us assume that $T = U \left| T \right| \in B \left( \mathbb{H} \right)$ be the polar decomposition of invertible bounded linear operator $T$. Then $\left| T \right| = \left( T^* T \right) ^{\frac{1}{2}}$ and $\left| T \right| ^{\frac{1}{2}}$ are invertible bounded linear operators either.
Moreover, since $T$ is bounded, we have
\begin{align*}
    \left \| T \left( x \right) - T \left( y \right) \right \| = \left \| T \left( x-y \right) \right \| \le \left \| T \right \| \left\| x-y \right\|, \quad \forall x, y \in \mathbb{H}.
\end{align*}
Similarly,
\begin{align*}
    \left \| \left| T \right| ^{\frac{1}{2}} \left( x \right) - \left| T \right| ^{\frac{1}{2}} \left( y \right) \right \| = \left \| \left| T \right| ^{\frac{1}{2}} \left( x-y \right) \right \| \le \left \| \left| T \right| ^{\frac{1}{2}} \right \| \left \| x-y \right\|, \quad \forall x,y \in B \left( \mathbb{H} \right).
\end{align*}
So, both $T$ and $\left| T \right|^{\frac{1}{2}}$ satisfy Lipschitz condition and turn out to be Lipeomorphism. Additionally,
\begin{align*}
    &\Delta \left( T \right) = \left| T \right|^{\frac{1}{2}} U \left| T \right|^{\frac{1}{2}} = \left| T \right|^{\frac{1}{2}} \left( U \left|T \right| \right) \left| T \right| ^{-\frac{1}{2}};\\
    &T = U \left| T \right| = \left| T \right|^{-\frac{1}{2}} \left( \left| T \right| ^{\frac{1}{2}} U \left| T \right|^{\frac{1}{2}} \right) \left| T \right|^{\frac{1}{2}}.
\end{align*}

Now, let us suppose that $T$ has bounded shadowing property. Using Lemma \ref{lem lipeomorphism}, we completed the proof and got the first result.

\end{proof}
In fact, the spectral picture of bounded linear operator $T \in B \left( \mathbb{H} \right)$ under Aluthge transformation has been described in \cite{IBJ2}. It can be observed that not only the spectrum $\sigma \left( T \right)$ but also its components $\sigma_p \left( T \right)$ (point spectrum of $T$) and $\sigma_{ap} \left( T \right)$ (appropriate spectrum of $T$) are preserved. We now show that the spectral properties of $T$ are preserved under $\lambda$-Aluthge transforms for all $\lambda \in \left( 0,1 \right)$.
\begin{thm}
\label{theorem spectral property of lambda Aluthge}
Let $T= U \left| T \right| \in B \left( \mathbb{H} \right)$ be the polar decomposition of $T$. For every $\lambda \in \left( 0,1\right)$, we suppose that the $\lambda$-Aluthge transform of $T$ is $\Delta_{\lambda} \left( T \right) = \left| T \right| ^{\lambda} U \left| T \right| ^{1-\lambda}$. Then the following properties held.
\begin{enumerate}[label=(\roman*)]
    \item $\sigma \left(T \right) = \sigma \left( \Delta_{\lambda} \left( T \right) \right)$;
    \item $\sigma _{ap} \left( T \right) = \sigma _{ap} \left( \Delta _{\lambda} \left( T\right) \right)$.
\end{enumerate}
\end{thm} 
\begin{proof}
The equality $\sigma \left(T \right) = \sigma \left( \Delta_{\lambda} \left( T \right) \right)$ was proved before in \cite[Lemma 5]{Huruya1}.\\
Let us assume that $T = U \left| T \right| \in B \left( \mathbb{H} \right)$ is a bounded linear operator on Hilbert space $\mathbb{H}$. For arbitrary $\lambda \in \mathbb{C}$, the $\lambda$-Aluthge transform of $T$ is determinded by
\begin{linenomath}
\begin{align*}
    \Delta _{\lambda} \left( T \right) = \left| T \right| ^{\lambda} U \left| T \right| ^{1-\lambda}.
\end{align*}
\end{linenomath}
Firstly, $\sigma _{ap} \left(T \right) \setminus \left \{ 0 \right \}= \sigma_{ap} \left( \Delta_{\lambda} \left(T \right) \right) \setminus \left \{ 0 \right \}$ as a result of Theorem 1 in \cite{Harte1}.\\
We now suppose that $0 \in \sigma_{ap} \left( T \right)$. By definition, there exists a sequence of unit vectors $\left \{ x_n \right \}_{n \rightarrow \infty}$ with $\left \| x_n \right \| =1$ (for every $n \in \mathbb{N}$) satisfying
\begin{linenomath}
\begin{equation*}
    \lim_{n \rightarrow \infty} \left \| U \left| T \right| x_n  \right \| =0.
\end{equation*}
\end{linenomath}
Consequently, $\left \| \left| T \right| ^{1-\lambda} x _n \right \| \rightarrow 0$ as $n \rightarrow \infty$ for every $\lambda \in \left( 0,1\right)$. It follows that 
\begin{linenomath}
\begin{align*}
    \Delta_{\lambda} \left( T \right) = \left \| \left| T \right| ^{\lambda} U \left| T \right| ^{1-\lambda} x_n \right \| \rightarrow 0 \quad \text{as } n \rightarrow \infty.
\end{align*}
\end{linenomath}
That means $0 \in \sigma _{ap} \left( \Delta _{\lambda} \left( T \right) \right)$ for every $\lambda \in \left( 0,1 \right)$.\\
Therefore, $\sigma_{ap} \left( T \right) \subseteq \sigma _{ap}\left( \Delta _{\lambda}\left( T\right) \right)$ for all $\lambda \in \left( 0,1 \right)$.\\
On the other sides, we suppose that $0 \in \sigma_{ap} \left( \Delta _{\lambda} \left( T \right) \right)$ for every $\lambda \in \left( 0,1\right)$. By hypothesis, there is a sequence $\left \{ x_n \right \} _{n \in \mathbb{Z}}$ of unit vectors in $\mathbb{H}$ so that
\begin{linenomath}
\begin{align*}
    \lim_{n \rightarrow \infty} \left \| \left| T \right| ^{\lambda} U \left| T \right|^{1-\lambda} x_n \right \| =0.
\end{align*}
\end{linenomath}
There are two subcases. \\
Firstly, for an arbitrary $\lambda \in \left( 0,1 \right)$, if $\left \| \left| T\right| ^{1-\lambda}x_n \right \| \rightarrow 0 $ then obviously $\left \| Tx_n \right \| = \left \| U \left| T \right|^{\lambda} \left| T\right|^{1-\lambda}x_n \right \| \rightarrow 0$ as $n \rightarrow \infty$. Secondy, for every $\lambda \in \left( 0,1 \right)$, if $\left\| \left| T \right|^{1-\lambda} x_n \right\|$ does not converge to $0$ then $\left \| U \left| T \right| ^{1-\lambda} \right \|$ does not converge to $0$ as $n \rightarrow \infty$. However, $\left| T \right|^{1-\lambda}$ maps $\left \{ U \left| T \right|^{1-\lambda} x_n \right \}$  to a null sequence in norm so $\left \| T \left( U \left| T  \right|^{1-\lambda}x_n \right) \right \| \longrightarrow 0$ when $n \longrightarrow \infty$ as a consequence. Since $U$ is an isometry, we get $\left\|U \left| T \right|^{1-\lambda} x_n \right \| =1$. Hence, $0 \in \sigma_{ap}\left( T \right)$ and $\sigma_{ap} \left( \Delta _{\lambda}\left( T \right) \right) \subseteq \sigma _{ap} \left( T \right)$.\\
Totally, the proof has been completed.
\end{proof}
As a consequence, the invariance of quasi-hyperbolicity of $T$ under the action of $\lambda$-Aluthge transforms for all $\lambda\in \left( 0,1 \right)$ is excuted.
\begin{defn}
\label{def quasi-hyperbolic}
Let $T \in B \left( X \right)$ be an invertible operator on a Banach space $X$. $T$ is said to be a hyperbolic operator if $\sigma \left( T \right) \cap \mathbb{T} = \varnothing$ in which $\mathbb{T}$ is the unit circle of complex plane $\mathbb{C}$.\\
We shall say that $T\in B \left( \mathbb{H} \right)$ is quasi-hyperbolic if there exists $n \in \mathbb{N}$ (independent of $x$) such that
\begin{linenomath}
\begin{align*}
    \max \left( \left \| T^{2n} x \right \|, \left \| x \right \| \right) \ge 2 \left \| T^{n} x \right \|, \quad \text{for all } x \in \mathbb{H}.
\end{align*}
\end{linenomath}
\end{defn}
\begin{prop}
\label{prop quasihyperbolic is preserved}
For a bounded linear operator $T \in B \left( \mathbb{H}\right)$, $T$ is quasi-hyperbolic if and only if its $\lambda$-Aluthge transform is for every $\lambda \in \left( 0,1 \right)$.
\end{prop}
\begin{proof}
C. J. K. Batty and Y. Tomilov showed in \cite{Batty1} that $T$ is quasi-hyperbolic if and only if $\sigma_{ap} \left( T \right)$ does not intersect with the unit circle $\mathbb{T}$ on complex plane $\mathbb{C}$. In other words,
\begin{linenomath}
\begin{align*}
    \sigma_{ap} \left( T \right) \cap \mathbb{T} = \varnothing.
\end{align*}
\end{linenomath}
For every $\lambda \in \left( 0,1 \right)$, the equality $\sigma_{ap} \left( T \right) = \sigma_{ap} \left( \Delta _{\lambda}\left( T \right) \right)$ in \ref{remark spectral property} provides $\sigma_{ap} \left( \Delta_{\lambda} \left( T \right) \right) \cap \mathbb{T} = \varnothing$. Thus $\Delta_{\lambda} \left( T \right)$ is quasi-hyperbolic for every $\lambda\in\left(0,1\right)$.
\end{proof}

\end{document}